\documentclass[11pt]{article}
\usepackage{xcolor}
\usepackage{mathrsfs}
\usepackage{amsmath,amsthm,amssymb,amscd}
\usepackage{booktabs} 
\usepackage{float}
\usepackage[hyperfootnotes=false, colorlinks, linkcolor={blue}, citecolor={magenta}, filecolor={blue}, urlcolor={blue}, plainpages=false, pdfpagelabels]{hyperref}
\usepackage[overload]{textcase} 
\usepackage{mathtools}
\usepackage[numbers,sort&compress]{natbib} 
\usepackage{etoolbox} 
\apptocmd{\sloppy}{\hbadness 10000\relax}{}{} 
\usepackage[left=2.5cm, right=2.5cm, top=3cm]{geometry}
\allowdisplaybreaks
\usepackage{enumerate}
\usepackage[shortlabels]{enumitem}
\usepackage{bm}
\usepackage{url}
\usepackage[multiple]{footmisc}

\usepackage{colonequals}
\usepackage{longtable} 
\usepackage[ampersand]{easylist}
\usepackage{ltablex,booktabs}

\def\re{\operatorname{Re}}
\def\spin{\operatorname{spin}}

\def\det{\operatorname{det}}
\def\Tr{\operatorname{Tr}}

\newcommand{\A}{{\mathbb A}}
\newcommand{\Q}{{\mathbb Q}}
\newcommand{\Z}{{\mathbb Z}}
\newcommand{\R}{{\mathbb R}}
\newcommand{\C}{{\mathbb C}}

\newcommand{\GL}{{\rm GL}}

\newcommand{\SL}{{\rm SL}}
\newcommand{\SO}{{\rm SO}}
\newcommand{\Sp}{{\rm Sp}}
\newcommand{\GSp}{{\rm GSp}}
\newcommand{\PGSp}{{\rm PGSp}}

\newcommand{\forget}[1]{}
\def\qdots{\mathinner{\mkern1mu\raise0pt\vbox{\kern7pt\hbox{.}}\mkern2mu
\raise3.4pt\hbox{.}\mkern2mu\raise7pt\hbox{.}\mkern1mu}}

\newtheorem{lemma}{Lemma}[section]
\newtheorem{theorem}[lemma]{Theorem}
\newtheorem{corollary}[lemma]{Corollary}
\newtheorem{proposition}[lemma]{Proposition}

\newtheorem{remark}{Remark}

\newcommand\blfootnote[1]{%
	\begingroup
	\renewcommand\thefootnote{}\footnote{#1}%
	\addtocounter{footnote}{-1}%
	\endgroup
}
\usepackage[toc,page, title, titletoc]{appendix}
\makeatother
\makeatletter
\newcommand\appendix@section[1]{%
	\refstepcounter{section}%
	\orig@section*{Appendix \@Alph\c@section: #1}%
	\addcontentsline{toc}{section}{Appendix \@Alph\c@section: #1}%
}
\g@addto@macro\appendix{\let\section\appendix@section}
\let\orig@section\section
\makeatother
\makeatletter
\def\thickhline{%
  \noalign{\ifnum0=`}\fi\hrule \@height \thickarrayrulewidth \futurelet
   \reserved@a\@xthickhline}
\def\@xthickhline{\ifx\reserved@a\thickhline
               \vskip\doublerulesep
               \vskip-\thickarrayrulewidth
             \fi
      \ifnum0=`{\fi}}
\makeatother

\title{On distinguishing Siegel cusp forms of degree two}
\author{Zhining Wei and Shaoyun Yi}
\date{}

\begin{document}

\maketitle
\blfootnote{2020 Mathematics Subject Classification: Primary 11F46, 11F60, 11F66 \\ \hspace*{0.22in} Key words and phrases. Siegel modular forms; Hecke eigenforms; $L$-functions; Rankin-Selberg method.}
\begin{abstract}
In this work, we establish several results on distinguishing Siegel cusp forms of degree two. In particular, a Hecke eigenform of level one can be determined by its second Hecke eigenvalue under a certain assumption. Moreover, we can distinguish two Hecke eigenforms of level one by using $L$-functions.
\end{abstract}

\tableofcontents

\section{Introduction}\label{sect intro}
One of the fundamental problems in the theory of automorphic forms is \textit{whether we can distinguish them by a set of eigenvalues}. It is well known that in the elliptic modular forms case, this question is equivalent to asking how many Fourier coefficients are sufficient to determine a normalized eigenform. This question is answered first by
the classical result of Sturm \cite{Sturm1987}. In 2011, Ghitza \cite{Gh2011} obtains a result by considering two cuspidal Hecke eigenforms of distinct weights, which improves a result of Ram Murty \cite{RamMurty1997}. Later, Vilardi and Xue \cite{VilardiXue2018} give a much stronger result that two normalized eigenforms of full level can be determined by their second coefficients under the assumption of Maeda's conjecture for the Hecke operator $T(2)$. Recently, Xue and Zhu \cite{XueZhu2022} generalize this result in terms of their third coefficients under a similar assumption.

However, distinguishing Siegel cusp forms is a long-standing unanswered problem and only recently Schmidt \cite{Schmidt2018}, in a remarkable paper, gives an affirmative answer to this question for normalized eigenvalues of a Siegel cuspidal eigenform of degree two. This result has been improved by Kumar, Meher and Shankhadhar \cite{KumarMeherShankhadhar2021} in the full level case, in which they essentially show that any set of eigenvalues (normalized or non-normalized) at primes $p$ of positive upper density are sufficient to determine the Siegel cuspidal eigenform. In this work, we further investigate the question on distinguishing Siegel cusp forms of degree two from various aspects with several improved results. We point out that the similar question for paramodular forms has been studied in \cite{WangWeiYanYi2023SMO} via the combination of the methods from both of automorphic side and Galois side.

Let $\mathcal{S}_{k}(\Gamma_0(N))$ be the space of Siegel cusp forms of level $\Gamma_0(N)$ and weight $k$, where $\Gamma_0(N)$ is the Siegel congruence subgroup of level $N$ defined as in \eqref{Siegel congruence subgroup of level N}. Let $F\in \mathcal{S}_{k}(\Gamma_0(N))$ be a Hecke eigenform with eigenvalue $\lambda_F(n)$ for $(n, N)=1$. Then our first main result is as follows. 
\begin{theorem}\label{distinguish-weight}
Let $k_1,k_2$ be distinct positive integers. Let $F\in\mathcal{S}_{k_1}(\Gamma_0(N))$ and $G\in\mathcal{S}_{k_2}(\Gamma_0(N))$ be Hecke eigenforms. Then we can find $n$ satisfying 
\begin{equation}
n\leq (2\log N+2)^4
\end{equation}
such that $\lambda_F(n)\neq\lambda_G(n)$.
\end{theorem}

\begin{remark}
It is shown in \cite[Corollary~5.3]{GhSa2014} that there exists some $n$ satisfying $n\leq  (2\log N+2)^6$ such that $\lambda_F(n)\neq\lambda_G(n)$.
In particular, we obtain an improved bound for $n$ in Theorem~\ref{distinguish-weight}. 
\end{remark}

Next, we assume that $N=1$, and let $\Gamma_2=\Sp(4, \Z)$. It is well known that the space $\mathcal{S}_k(\Gamma_2)$ has a natural decomposition into orthogonal subspaces
\begin{equation}\label{type decomposition}
    \mathcal{S}_k(\Gamma_2)=\mathcal{S}_k^{\mathbf{(P)}}(\Gamma_2)\oplus \mathcal{S}_k^{\mathbf{(G)}}(\Gamma_2)
\end{equation}
with respect to the Petersson inner product. Here, $\mathcal{S}_k^{\mathbf{(P)}}(\Gamma_2)$ is the subspace of Saito-Kurokawa liftings, and $\mathcal{S}_k^{\mathbf{(G)}}(\Gamma_2)$ is the subspace of non-liftings. We refer the reader to \cite[\S~2.1]{Schmidt2018} for more details about this type decomposition. For our purpose, let
\begin{equation}\label{weak Maeda’s conjecture}
    \mathcal{K}^{\mathbf{(*)}}(2)\colonequals\{k\in\Z\colon \text{The characteristic polynomial
of } T_k(2) \text{ for } F\in\mathcal{S}_k^{\mathbf{(*)}}(\Gamma_2)  \text{ is irreducible}\},
\end{equation}
where $\mathcal{S}_k^{\mathbf{(*)}}(\Gamma_2)$ is the set of those $F\in \mathcal{S}_k(\Gamma_2)$ of type $\mathbf{(*)}\in\{\mathbf{(P)}, \mathbf{(G)}\}$ as in \eqref{type decomposition}. Then we can prove the following result.
\begin{theorem}\label{disti-maeda}
Let $k_1, k_2\in\mathcal{K}^{\mathbf{(P)}}(2)\cap \mathcal{K}^{\mathbf{(G)}}(2)$ be two even positive integers, where $k_1$ and $k_2$ may equal. Let $F\in\mathcal{S}_{k_1}(\Gamma_2)$ and $G\in\mathcal{S}_{k_2}(\Gamma_2)$ be Hecke eigenforms. If $\lambda_F(2)=\lambda_G(2)$, then $F=c\cdot G$ for some non-zero constant $c$.
\end{theorem}

\begin{remark}
The set $\mathcal{K}^{\mathbf{(*)}}(2)$ defined as in \eqref{weak Maeda’s conjecture} is basically a weak version of the generalized Maeda’s conjecture for $\Gamma_2=\Sp(4, \Z)$. In fact, Maeda's conjecture for $\Gamma_1=\SL(2, \Z)$ would imply that $\mathcal{K}^{\mathbf{(P)}}(2)=\{k\colon \mbox{$k$ even and $k\geq 10$}\}$. Moreover, it is expected that the set $\mathcal{K}^{\mathbf{(G)}}(2)$ has the natural density of $1$. See \cite{HidaMaeda1997} for more discussions about the Maeda's conjecture for $\SL(2, \Z)$.

In this work, we focus on the Hecke operator $T(2)$. The Maeda's conjecture for $T(2)$ is verified numerically for elliptic cusp forms of level one and weight $k\leq 14,000$ by Ghitza and McAndrew; see \cite[Table~1]{GhitzaMcAndrew2012}. This implies that $\mathcal{K}^{\mathbf{(P)}}(2)$ contains even integers less than $7,000$. In his master thesis, McAndrew has verified the Madea's conjecture for Siegel cusp forms of level one and even weight $k$ with $k\in[28,110]$; see \cite[Theorem~5.22]{mcandrew2013maeda}.  
\end{remark}

In addition, we can also distinguish Hecke eigenforms in each type by using $L$-functions with different methods. First, recall that Saito-Kurokawa liftings of level one and weight $k\in 2\Z_{>0}$ can be obtained from elliptic cusp forms of level one and weight $2k-2$. More precisely, let $f\in\mathcal{S}_{2k-2}(\Gamma_1)$ be a Hecke eigenform, and let $\pi_f$ be the cuspidal automorphic  representation of $\GL(2, \A)$ associated to $f$. Here, $\A$ is the ring of adeles of $\Q$. Then the resulting Saito-Kurokawa lifting is in $\mathcal{S}_k(\Gamma_2)$, denoted by $F_f$. It is well known that $F_f$ is also a Hecke eigenform; see \cite{Kurokawa1978, Maass1979} for more details about the classical Saito-Kurokawa liftings. The normalized spinor $L$-function of $F_f$ and the normalized $L$-function of $f$ are connected by the following relation:
\begin{equation}
    L(s,\pi_{F_f},\rho_4)=\zeta(s+1/2)\zeta(s-1/2)L(s,\pi_f),
\end{equation}
where $\rho_4$ is the $4$-dimensional irreducible representation of $\Sp(4, \C)$, and $\pi_{F_f}$ is the cuspidal automorphic  representation of $\GSp(4, \A)$ corresponding to $F_f$. Let $\xi$ be a primitive Dirichlet character, and let $\chi$ be the corresponding Hecke character of $\Q^\times\backslash\A^\times$. Let $\sigma_1$ be the standard representation of the dual group $\GL(1, \C)=\C^\times$. Then we can define the twisted spinor $L$-function by 
\begin{equation}\label{twisted L-function for SK lift}
    L(s,\pi_{F_f}\times\chi,\rho_4\otimes\sigma_1)=L(s+1/2,\chi)L(s-1/2,\chi)L(s,\pi_f\times\chi).
\end{equation}
Using \cite[Theorem~B]{LR1997}, we can show the following result for Saito-Kurokawa liftings of full level.
\begin{proposition}\label{Detemination of SK}
Let $k_1,k_2$ be even positive integers and $f\in\mathcal{S}_{2k_1-2}(\Gamma_1), g\in\mathcal{S}_{2k_2-2}(\Gamma_1)$ be normalized Hecke eigenforms. Suppose that there exists a non-zero constant $c$ such that
\begin{equation}\label{main assumption for G type Thm 3}
L(1/2,\pi_{F_f}\times\chi_d,\rho_4\otimes\sigma_1)=c\cdot L(1/2,\pi_{F_g}\times\chi_d,\rho_4\otimes\sigma_1)
\end{equation}
for almost all quadratic Hecke characters $\chi_d$ of $\Q^\times\backslash\A^\times$, which are corresponding to primitive quadratic Dirichlet characters $\xi_d$ of conductor $d$. Then $k_1=k_2$ and $F_f=F_g$. 
\end{proposition}

Finally, we will distinguish Hecke eigenforms of type $\mathbf{(G)}$ (i.e., non-liftings) by using Rankin-Selberg $L$-functions under the Generalized Riemann Hypothesis. Let $k_1, k_2$ be even integers. Let $F\in\mathcal{S}_{k_1}^{\mathbf{(G)}}(\Gamma_2)$ and $G\in\mathcal{S}_{k_2}^{\mathbf{(G)}}(\Gamma_2)$ be Hecke eigenforms, and let $\pi_F$ (resp. $\pi_G$) be the cuspidal automorphic  representation of $\GSp(4, \A)$ corresponding to $F$ (resp. $G$). Then we can define the Rankin-Selberg $L$-function of $F$ and $G$, denoted by $L(s,\pi_F\times \pi_G,\rho_i\otimes\rho_j)$ with $i,j\in\{4,5\}$; see \cite[(271)]{PiSaSc2014}. Note that the $L$-functions here are actually the finite part of $L$-functions in \cite{PiSaSc2014}; see Appendix~\ref{Archimedean factors append} for the explicit forms of archimedean factors associated to these \texorpdfstring{$L$}{L}-functions. Moreover, $L(s,\pi_F\times \pi_G,\rho_i\otimes\rho_j)$ has a simple pole at $s=1$ if and only if $i=j, k_1=k_2$ and $F=c\cdot G$ for some non-zero constant $c$; for example see \cite[Theorem~5.2.3]{PiSaSc2014}. In this case, we will apply the method in \cite{GoHo1993} and Lemma~\ref{rkselemma} to obtain the following result:
\begin{theorem}\label{dis-non}
Let $k_1, k_2$ be even integers. Let $F\in\mathcal{S}_{k_1}^{\mathbf{(G)}}(\Gamma_2)$ and $G\in\mathcal{S}_{k_2}^{\mathbf{(G)}}(\Gamma_2)$ be Hecke eigenforms, and let $\pi_F$ (resp. $\pi_G$) be the cuspidal automorphic  representation of $\GSp(4, \A)$ corresponding to $F$ (resp. $G$). Suppose that $L(s,\pi_F\times \pi_G,\rho_4\otimes\rho_4)$ and $L(s,\pi_F\times \pi_F,\rho_4\otimes\rho_4)$ satisfy the Generalized Riemann Hypothesis. If $F$ is not a scalar multiplication of $G$, then there exists an integer 
\begin{equation}
n\ll(\log k_1k_2)^2(\log\log k_1k_2)^4
\end{equation}
such that $\tilde{\lambda}_F(n)\neq \tilde{\lambda}_G(n)$. Here, $\tilde{\lambda}_F(n)=n^{3/2-k_1}\lambda_F(n)$ (resp. $\tilde{\lambda}_G(n)=n^{3/2-k_2}\lambda_G(n)$) is the normalized Hecke eigenvalue for $F$ (resp. $G$).
\end{theorem}
\section*{Acknowledgements}
The authors would like to thank the referee for the careful revision and a number of insightful suggestions. The authors would also like to thank Wenzhi Luo, Kimball Martin, Ralf Schmidt, Biao Wang, Pan Yan and Liyang Yang for their helpful discussions and comments. The authors further thank Ariel Weiss for drawing our attention to the low weights case ($k_i=1, 2$) in Theorem~\ref{distinguish-weight} and Biplab Paul for forwarding us their paper \cite{GunKohnenPaul2021}. Shaoyun Yi is supported by the National Natural Science Foundation of China (No. 12301016) and the Fundamental Research Funds for the Central Universities (No. 20720230025).
\section{Preliminaries}
We consider the symplectic similitude group
\begin{equation}
\GSp(4) \coloneqq \{g\in\GL(4)\colon \:^tgJg=\mu(g)J,\:\mu(g)\in \GL(1)\},
\end{equation}
which is an algebraic $\Q$-group. Here, $J=\begin{bsmallmatrix} &&&1\\&&1&\\&-1&&\\-1&&&\end{bsmallmatrix}$. The function $\mu$ is called the multiplier homomorphism. The kernel of this function is the symplectic group $\Sp(4)$. Let $\mathrm{Z}$ be the center of $\GSp(4)$ and $\PGSp(4)=\GSp(4)/\mathrm{Z}$. When speaking about Siegel modular forms of degree two, it is more convenient to realize symplectic groups using the symplectic form $J=\begin{bsmallmatrix} 0&1_2\\-1_2&0\end{bsmallmatrix}$. The Siegel upper half plane of degree 2 is defined by 
\begin{equation}
\label{eq-H2}
\mathbb{H}_2\coloneqq\{Z\in\mathrm{Mat}_2(\C)\colon \,^tZ=Z, \mathrm{Im}(Z)>0\}.
\end{equation}
The group $\GSp(4, \mathbb{R})^+\coloneqq\{g\in\GSp(4, \R)\colon \mu(g)>0\}$ acts on $\mathbb{H}_2$ by
\begin{equation}
    g\langle Z\rangle\coloneqq (AZ+B)(CZ+D)^{-1}\quad \text{for } g=\begin{bsmallmatrix} A&B\\C&D\end{bsmallmatrix}\in \GSp(4, \mathbb{R})^+ \text{ and } Z\in\mathbb{H}_2.
\end{equation}
Let $\Gamma_2=\Sp(4, \Z)$. In general, for a positive integer $N$ we let 
\begin{equation}\label{Siegel congruence subgroup of level N}
    \Gamma_0(N)\coloneqq\left\{\begin{bsmallmatrix} A&B\\C&D\end{bsmallmatrix}\in\Sp(4, \Z)\colon C\equiv 0\pmod{N}\right\}
\end{equation}
be the Siegel congruence subgroup of level $N$. In particular, $\Gamma_2=\Gamma_0(1)$.

Let $\mathcal{M}_k(\Gamma_0(N))$ be the space of Siegel modular forms of weight $k$ with respect to $\Gamma_0(N)$, and let $\mathcal{S}_k(\Gamma_0(N))$ be the subspace of cusp forms. That is to say, for any function $F\in \mathcal{M}_k(\Gamma_0(N))$, it is a holomorphic $\C$-valued function on $\mathbb{H}_2$ satisfying $\big(F|_k\gamma\big)(Z)=F(Z)$ for all $\gamma\in\Gamma_0(N)$. Here, 
\begin{equation}
    \big(F|_k g\big)(Z)\coloneqq\mu(g)^kj(g, Z)^{-k}F(g\langle Z\rangle)\quad \text{for } g=\begin{bsmallmatrix} A&B\\C&D\end{bsmallmatrix}\in \GSp(4, \mathbb{R})^+ \text{ and } Z\in\mathbb{H}_2,
\end{equation}
where $j(g, Z)\coloneqq \det(CZ+D)$ is the automorphy factor. We remark that this operator differs from the classical one used in \cite{Andrianov1974} by a factor. We do so to make the center of $\GSp(4, \mathbb{R})^+$ act trivially.

Let $F\in\mathcal{S}_k(\Gamma_0(N))$ be a Hecke eigenform, i.e., it is an eigenvector for all the Hecke operators $T(n), (n, N)=1$. Denote by $\lambda_F(n)$ the eigenvalue of $F$ under $T(n)$ when $(n,N)=1$. For any prime $p\nmid N$, we let $\alpha_{p,0},\alpha_{p,1},\alpha_{p,2}$ be the classical Satake parameters of $F$ at $p$. It is well known that
\begin{equation}
    \alpha_{p,0}^2\alpha_{p,1}\alpha_{p,2}=p^{2k-3}.
\end{equation}
In particular, let $N=1$ and $F\in\mathcal{S}_k(\Gamma_2)$ be a Hecke eigenform, we can define the $L$-series
\begin{equation}\label{defn of Hs}
    H(s)=\sum_{n=1}^{\infty}\frac{\lambda_F(n)}{n^s}.
\end{equation}
This can be written as a Euler product 
\begin{equation}
    H(s)=\prod_pH_p(s)=\prod_p\left(1+\frac{\lambda_F(p)}{p^s}+\frac{\lambda_F(p^2)}{p^{2s}}+\cdots\right)
\end{equation}
provided $\re(s)>k$. Moreover, one can show that
\begin{equation}\label{eigenspin}
    H_p(s)=\big(1-p^{2k-4-2s}\big)L_p(s,F,\spin),
\end{equation}
where $L_p(s,F,\spin)$ is the local spinor $L$-factor of $F$ at $p$ and it can be given by
\begin{equation}\label{localspin}
  L_p(s,F,\spin)^{-1}=(1-\alpha_{p,0}p^{-s})(1-\alpha_{p,0}\alpha_{p,1}p^{-s})(1-\alpha_{p,0}\alpha_{p,2}p^{-s})(1-\alpha_{p,0}\alpha_{p,1}\alpha_{p,2}p^{-s}).
\end{equation}
On the other hand, by \cite[pp.~62, 69]{Andrianov1974} one can see that
\begin{equation}\label{eisp}
L_p(s,F,\spin)^{-1}=1-\lambda_F(p)p^{-s}+(\lambda_F(p)^2-\lambda_F(p^2)-p^{2k-4})p^{-2s}-\lambda_F(p)p^{2k-3-3s}+p^{4k-6-4s}.
\end{equation}
As a consequence, we can define the spinor $L$-function 
\begin{equation}
L(s,F,\spin)=\prod_{p}L_p(s,F,\spin).
\end{equation}
Let $\alpha_p=p^{3/2-k}\alpha_{p,0}$ and $\beta_p=\alpha_p\alpha_{p,1}$. By comparing \eqref{localspin} with \eqref{eisp}, we obtain (also see \cite[Proposition~4.1]{PiSch2009})
\begin{align}
\lambda_F(p)&=p^{k-3/2}(\alpha_p+\alpha_p^{-1}+\beta_p+\beta_p^{-1})\label{lambda p},\\
\lambda_F(p^2)&=p^{2k-3}\left((\alpha_p+\alpha_p^{-1})^2+(\alpha_p+\alpha_p^{-1})(\beta_p+\beta_p^{-1})+(\beta_p+\beta_p^{-1})^2-2-1/p\right).\label{lambda p2}    
\end{align}    

Let $\rho_4$ be the $4$-dimensional irreducible representation of $\Sp(4, \C)$. In fact, $\rho_4$ is the natural representation of $\Sp(4, \C)$ on $\C^4$, which is also called the spin representation. For later use, we would normalize the spinor $L$-functions such that they satisfy a functional equation relating $s$ and $1-s$. More precisely, the normalized spinor $L$-function $L(s,\pi_F,\rho_4)$ is defined as follows
\begin{equation}\label{nor-spin}
    L(s,\pi_F,\rho_4)=L(s+k-3/2,F,\spin)=\sum_{n=1}^{\infty}\frac{a_F(n)}{n^s}.
\end{equation}
Note that this is the finite part of the completed $L$-function of $\pi_F$, where $\pi_F$ is the cuspidal automorphic  representation of $\GSp(4, \A)$ associated to $F$. See \cite{AsgariSchmidt2001} and \cite[Section~4.2]{Schmidt2017} for more details about the connection between Siegel modular forms of degree two and automorphic representations of $\GSp(4, \A)$. Moreover, let $\tilde{\lambda}_F(n)=n^{3/2-k}\lambda_F(n)$ be the normalized eigenvalues. It follows from \eqref{defn of Hs} and \eqref{eigenspin} that
\begin{equation}\label{nor-ei-spin}
    \sum_{n=1}^{\infty}\frac{\tilde{\lambda}_F(n)}{n^s}=\zeta(2s+1)^{-1}L(s,\pi_F,\rho_4).
\end{equation}
Here, $\zeta$ is the Riemann zeta function. 

If $F\in\mathcal{S}_k(\Gamma_0(N))$ with level $N>1$, we still can define the partial spinor $L$-functions by Euler products for all primes $p$ not dividing $N$. In particular, the local factor at $p$ with $(p, N)=1$ is defined in the same way as above.

Similarly, let $\rho_5$ be the $5$-dimensional irreducible representation of $\Sp(4, \C)$. An explicit formula for the representation $\rho_5$ as a map $\Sp(4, \C)\to \SO(5, \C)$ is given in \cite[Appendix~A.7]{RobertsSchmidt2007}. The standard $L$-function associated to $F$ is defined as
\begin{equation}
    L(s,\pi_F,\rho_5)=\prod_p L_p(s,F,\mathrm{std})=\sum_{n=1}^{\infty}\frac{b_F(n)}{n^s},
\end{equation}
where
\begin{equation}\label{local standard L function}  
  L_p(s,F,\mathrm{std})^{-1}=(1-p^{-s})(1-\alpha_{p,1}p^{-s})(1-\alpha_{p,2}p^{-s})(1-\alpha_{p,1}^{-1}p^{-s})(1-\alpha_{p,2}^{-1}p^{-s}).
\end{equation}
Again, for $F\in\mathcal{S}_k(\Gamma_0(N))$ with level $N>1$, we can define the partial standard $L$-functions by Euler products for all primes $p$ not dividing $N$ in the same way.
\section{Proof of Theorem~\ref{distinguish-weight}}
First, by \eqref{eigenspin} and \eqref{eisp} we have
\begin{equation}
\lambda_F(p^3)=2\lambda_F(p)\lambda_F(p^2)-\lambda_F(p)^3+\lambda_F(p)(p^{2k-3}+p^{2k-4}),\label{1.5}
\end{equation}
and
\begin{equation}\label{1.6}
    \lambda_F( p^4)=-\lambda_F( p)^4+\lambda_F( p)^2\lambda_F( p^2)+\lambda_F( p^2)^2+\lambda_F( p)^2p^{2k-4}+\lambda_F( p^2)p^{2k-4}+2\lambda_F( p)^2 p^{2k-3}-p^{4k-6}.
\end{equation}
Then we can show the following result:
\begin{theorem}\label{1.1}
Let $k_1,k_2$ be distinct positive integers. Let $F\in\mathcal{S}_{k_1}(\Gamma_0(N))$ and $G\in\mathcal{S}_{k_2}(\Gamma_0(N))$ be Hecke eigenforms. Then for any prime $p$ not dividing $N$, we can find $i\in\{1,2,3,4\}$ such that 
\[\lambda_F(p^i)\neq\lambda_G(p^i).\]
\end{theorem}
To prove this theorem, we also need the following lemma.
\begin{lemma}\label{1.4}
Let $F\in\mathcal{S}_k(\Gamma_0(N))$ be a Hecke eigenform with $k\in\Z_{>0}$. Let $p\nmid N$ be a prime. If $\lambda_F(p)=0$, then
\begin{equation}
|\lambda_F(p^2)|\leq p^{2k-2}+2p^{2k-4}.
\end{equation}
\end{lemma}
\begin{proof}
By \eqref{lambda p}-\eqref{lambda p2} and $\lambda_F(p)=0$, we have
\begin{align*}
\lambda_F(p^2)&=\lambda_F(p)^2-p^{2k-3}((\alpha_p+\alpha_p^{-1})(\beta_p+\beta_p^{-1})+2+1/p)\\
&=p^{2k-3}((\alpha_p+\alpha_p^{-1})(\alpha_p+\alpha_p^{-1})-2-1/p)\\
&=p^{2k-3}(\alpha_p^2+\alpha_p^{-2}-1/p)=p^{2k-3}(\beta_p^2+\beta_p^{-2}-1/p).
\end{align*}
Hence the desired assertion follows immediately from the bound of $\alpha_p, \beta_p$ as below:
\begin{equation}\label{Satake parameter estimate}
    p^{-1/2}\leq |\alpha_p|, |\beta_p|\leq p^{1/2}.
\end{equation}
To see \eqref{Satake parameter estimate}, we separate two cases: i) If $F$ is not of type \textbf{(G)}, then \eqref{Satake parameter estimate} follows from \cite[Table~1]{Schmidt2018} (note that the type \textbf{(F)} cannot occur) and the Ramanujan conjecture for elliptic modular forms; ii) If $F$ is of type \textbf{(G)}, then \eqref{Satake parameter estimate} follows from the fact that the cuspidal automorphic representation $\pi_F$ of $\PGSp(4, \A)$ associated to $F$ admits a functorial transfer to a unitary, cuspidal automorphic representation of $\GL(4,\A)$ and the Jacquet-Shalika bound for $\GL(4)$ (see \cite[Corollary~2.5]{JacquetShalika1981}).
\end{proof}

\begin{proof}[Proof of Theorem~\ref{1.1}]
Assume that there exists a prime $p\nmid N$ such that $\lambda_F(p^{i})=\lambda_G(p^i)$ for $i=1,2,3,4$; we will obtain a contradiction. More precisely, we consider the following two cases:

(1) If $\lambda_F(p)\neq 0$, then by \eqref{1.5} and $\lambda_F(p^i)=\lambda_G(p^i)$ ($i=1,2,3$), we have
\begin{equation}
\lambda_F(p)(p^{2k_1-3}+p^{2k_1-4})=\lambda_G(p)(p^{2k_2-3}+p^{2k_2-4}).
\end{equation}
This yields the contradiction $k_1=k_2$.

(2) If $\lambda_F(p)=0$, then $\lambda_G(p)=0$ by assumption. By Lemma~\ref{1.4} we have
\begin{equation}
|\lambda_F(p^2)|\leq p^{2k_1-2}+2p^{2k_1-4}\quad\text{and}\quad |\lambda_G(p^2)|\leq  p^{2k_2-2}+2p^{2k_2-4}.
\end{equation}
Without loss of generality, we assume that $k_1\geq k_2+1$. Since $\lambda_F( p^2)= \lambda_G( p^2)$, we have
\begin{equation}\label{est lambda p 2 equal}
    |\lambda_F( p^2)|=|\lambda_G( p^2)|\leq p^{2k_2-2}+2p^{2k_2-4}.
\end{equation}
On the other hand, it follows from \eqref{1.6} and $\lambda_F(p^i)=\lambda_G(p^i), i=1,2,3,4$, that
\begin{equation}\label{key eqn to get contradiction}
    \lambda_F( p^2)p^{2k_1-4}-p^{4k_1-6}=\lambda_G( p^2)p^{2k_2-4}-p^{4k_2-6}.
\end{equation}
Then we have $\lambda_F( p^2)(p^{2k_1-4}-p^{2k_2-4})=p ^{4k_1-6}-p ^{4k_2-6}$. Multiplying $p^2$ both sides we obtain
\begin{equation}
    \lambda_F( p^2)(p^{2k_1-2}-p^{2k_2-2})=p ^{4k_1-4}-p ^{4k_2-4}
=(p ^{2k_1-2}-p ^{2k_2-2})(p ^{2k_1-2}+p ^{2k_2-2}).
\end{equation}
It follows that $\lambda_F(p^2)=p ^{2k_1-2}+p ^{2k_2-2}$. This equality leads to a contradiction due to \eqref{est lambda p 2 equal} and $k_1\geq k_2+1$.
\end{proof}
Then Theorem~\ref{distinguish-weight} immediately follows from Theorem~\ref{1.1} and Lemma~\ref{useful lemm} below.
\begin{lemma}[\mbox{\cite[c.f.~\S~2]{Gh2011} }]\label{useful lemm}
Let $N\geq 1$ be a positive integer, then we can find a prime $p$ such that $(p,N)=1$ and $p\leq 2\log N+2$.
\end{lemma}

\section{Proof of Theorem~\ref{disti-maeda}}
In this section, we only consider $\Gamma_2=\Sp(4, \Z)$. Let $m_k=\dim_{\C}\mathcal{S}_k(\Gamma_2)$. For $\mathbf{(*)}\in\{\mathbf{(G)}, \mathbf{(P)}\}$, recall that the set $\mathcal{K}^{\mathbf{(*)}}(2)$ of integers is defined as in \eqref{weak Maeda’s conjecture}. Moreover, we let $m_{k}^{\mathbf{(*)}}=\dim_\C\mathcal{S}_{k}^{\mathbf{(*)}}(\Gamma_2)$. Evidently, $m_{k}=m_{k}^{\mathbf{(P)}}+m_{k}^{\mathbf{(G)}}$. It is well known that $m_{k}^{\mathbf{(P)}}>0$ only if $k\in \Z_{\geq 10}$ is even and $m_{k}^{\mathbf{(G)}}>0$ only if $k\in \Z_{\geq 20}$; for example see \cite[Theorem~3.1]{RoySchmidtYi2021}. We also note that $m_k^{(\mathbf{P})}=\dim_{\mathbb{C}}\mathcal{S}_{2k-2}(\Gamma_1)$.
\begin{proof}[Proof of Theorem~\ref{disti-maeda}] We are going to separate into the following three cases.

\textbf{Case I:} If $F$ and $G$ both are Saito-Kurokawa liftings, say $F\in \mathcal{S}_{k_1}^{\mathbf{(P)}}(\Gamma_2)$ and $G\in \mathcal{S}_{k_2}^{\mathbf{(P)}}(\Gamma_2)$ with $k_1, k_2\in \mathcal{K}^{\mathbf{(P)}}(2)$, then we can write $F=F_f$ and $G=F_g$, which are lifts from $f\in\mathcal{S}_{2k_1-2}(\Gamma_1)$ and $g\in\mathcal{S}_{2k_2-2}(\Gamma_1)$, respectively. Recall that if $f$ and $g$ are Hecke eigenforms, then both $F_f$ and $F_g$ are also Hecke eigenforms.
Let $T_{2k-2}^{(1)}(2)$ be the Hecke operator on $f\in\mathcal{S}_{2k-2}(\Gamma_1)$ with Hecke eigenvalue $\lambda_f(2)$, and let $T_k(2)$ be the Hecke operator on $F_f\in \mathcal{S}_k(\Gamma_2)$ with Hecke eigenvalue $\lambda_{F_f}(2)$. Then we have
\begin{equation}\label{eigenvalues of P type}
    \lambda_{F_f}(2)=2^{k-1}+2^{k-2}+\lambda_f(2).
\end{equation}
Moreover, let $P(T_k^{(\mathbf{P})}(2),t)$ be the characteristic polynomial of $T_k(2)$ on $\mathcal{S}_k^{\mathbf{(P)}}(\Gamma_2)$, which is irreducible if $k\in \mathcal{K}^{\mathbf{(P)}}(2)$. For $k\in \mathcal{K}^{\mathbf{(P)}}(2)$ we can see that the characteristic polynomial $P(T_{2k-2}^{(1)}(2),t)$ of $T_{2k-2}^{(1)}(2)$ is irreducible as well since $P(T_k^{(\mathbf{P})}(2),t)=P(T_{2k-2}^{(1)}(2),t-2^{k-1}-2^{k-2})$. We can further assume that  $f$ is not a constant multiple of $g$ since the Saito-Kurokawa lifting is injective.
\begin{enumerate}
    \item[I-i).] If $k_1=k_2=k$, then $\lambda_{F_f}(2)\neq\lambda_{F_g}(2)$ due to the fact that the irreducible characteristic polynomial $P(T_k^{(\mathbf{P})}(2),t)$ has distinct roots.
    \item[I-ii).] If $m_{k_1}^{(\mathbf{P})}\neq m_{k_2}^{(\mathbf{P})}$, then $\deg P(T_{k_1}^{(\mathbf{P})}(2),t)\neq \deg P(T_{k_2}^{(\mathbf{P})}(2),t)$. Recall that both of them are irreducible, it follows that $P(T_{k_1}^{(\mathbf{P})}(2),t)$ and $P(T_{k_2}^{(\mathbf{P})}(2),t)$ have distinct roots. Hence, $\lambda_{F_f}(2)\neq\lambda_{F_g}(2)$.
    \item[I-iii).] If $m_{k_1}^{(\mathbf{P})}=m_{k_2}^{(\mathbf{P})}\geq 1$ and $k_1\neq k_2$, it is clear that $2k_1-2,2k_2-2\geq 18$. Additionally, we can show that there exists $n\geq 1$ such that $2k_1-2,2k_2-2\in\{12n+6,12n+10,12n+14\}$ since $k_1, k_2$ are even and $m_{k_1}^{(\mathbf{P})}=m_{k_2}^{(\mathbf{P})}$. On the other hand, we can show that
\begin{equation}\label{trace of P type}
    \Tr T_{k_1}^{(\mathbf{P})}(2)=m^{(\mathbf{P})}_{k_1}(2^{k_1-1}+2^{k_1-2})+\Tr T_{2k_1-2}^{(1)}(2).
\end{equation}
Assume that $k_1=k_2+l$ with $l>0$. Then by the choice of $k_1,k_2$, we know that $l\in\{2,4\}$. Let $l=2^m$ as in \cite[Corollary~3.4]{VilardiXue2018}, and so $m\in \{1, 2\}$. It follows from $m_{k_1}^{(\mathbf{P})}=m_{k_2}^{(\mathbf{P})}$ and \eqref{trace of P type} that
\begin{equation*}
\Tr T_{k_1}^{(\mathbf{P})}(2)-\Tr T_{k_2}^{(\mathbf{P})}(2)
=2^{k_2-2}\left(m^{(\mathbf{P})}_{k_2}(2^{k_1-k_2+1}+2^{k_1-k_2}-3)+a_{k_2-1, l}-2^{m+5-k_2}c_{k_2-1, l}\right),    
\end{equation*}
where $a_{k_2-1,l}$ is an integer and $c_{k_2-1,l}$ is an odd integer as in the proof of \cite[Corollary~3.4]{VilardiXue2018}. However, we know that $m+5-k_2\leq -3$ since $m\leq 2$ and $k_2\geq 10$. Therefore, $2^{m+5-k_2}c_{k_2-1,l}$ is not an integer and hence $\Tr T_{k_1}^{(\mathbf{P})}(2)\neq\Tr T_{k_2}^{(\mathbf{P})}(2)$. By irreducibility of characteristic polynomials $\Tr T_{k_1}^{(\mathbf{P})}(2)=\Tr \lambda_{F_f}(2)$ and $\Tr T_{k_2}^{(\mathbf{P})}(2)=\Tr \lambda_{F_g}(2)$, which implies that $\lambda_{F_f}(2)\neq \lambda_{F_g}(2)$ as desired.
\end{enumerate}

\textbf{Case II:} If $F$ and $G$ both are non-liftings, say $F\in \mathcal{S}_{k_1}^{\mathbf{(G)}}(\Gamma_2)$ and $G\in \mathcal{S}_{k_2}^{\mathbf{(G)}}(\Gamma_2)$, then we can apply the similar arguments as in the above Case I. More precisely, we will separate the following three cases:
\begin{enumerate}
\item[II-i).] If $k_1=k_2=k\in \mathcal{K}^{\mathbf{(G)}}(2)$, as the characteristic polynomial of $T_{k}^{(\mathbf{G})}(2)$ is irreducible by assumption, then all of its roots are distinct. Thus if $F\neq c\cdot G$ for any non-zero constant $c$, then $\lambda_F( 2)\neq \lambda_G( 2)$. 
\item[II-ii).] If $k_1\neq k_2$, then it follows from straightforward computations by using \cite[Theorem~3.1]{RoySchmidtYi2021} that $m_{k_1}^{\mathbf{(G)}}>m_{k_2}^{\mathbf{(G)}}$ for any $k_1>k_2\geq 40$. In particular, we can conclude $\lambda_F( 2)\neq \lambda_G( 2)$ with the similar argument in the above I-ii).
\item[II-iii).] If $k_1\neq k_2$ and $m_{k_1}^{\mathbf{(G)}}=m^{\mathbf{(G)}}_{k_2}$, then we can just use \cite{BFvdGweb2017} to see that $\Tr T_{k_1}^{(\mathbf{G})}(2)\neq \Tr T_{k_2}^{(\mathbf{G})}(2)$ for all even weights $k_1, k_2\leq 40$. Hence the desired assertion follows.
\end{enumerate}

\textbf{Case III:} If one of $F$ and $G$ is a Saito-Kurokawa lifting and the other one is non-lifting, say $F\in \mathcal{S}_{k_1}^{\mathbf{(P)}}(\Gamma_2)$ and $G\in \mathcal{S}_{k_2}^{\mathbf{(G)}}(\Gamma_2)$. It follows from \eqref{lambda p} and \eqref{eigenvalues of P type} that if $k_1-k_2\geq 6$, then we must have $\lambda_{F}(2)>\lambda_G(2)$. Next, we only need to consider $k_1-k_2\leq 4$ cases. Again, by \cite[Theorem~3.1]{RoySchmidtYi2021} we can easily to see that $m_{k_1}^{\mathbf{(P)}}\neq m_{k_2}^{\mathbf{(G)}}$ unless $k_2\in S\coloneqq\{20, 22, 24, 26, 28,30, 32\}$. By \cite{Breulmann1999}, we know that the Hecke eigenvalues $\lambda_{F}(n)>0$ for all $n$. Then by irreducibility of characteristic polynomials $\Tr T_{k_1}^{\mathbf{(P)}}(2)=\Tr\lambda_{F}(2)>0$. On the other hand, for every $k_2\in S$, by \cite{BFvdGweb2017} we can see that $\Tr T_{k_2}^{\mathbf{(G)}}(2)<0$ and so $\Tr\lambda_G(2)=\Tr T_{k_2}^{\mathbf{(G)}}(2)\neq \Tr T_{k_1}^{\mathbf{(P)}}(2)$. In particular, we have $\lambda_F(2)\neq \lambda_G(2)$. Therefore we complete the proof of theorem.
\end{proof}
Since Saito-Kurokawa liftings only happen for even weights, there is no need to discuss the odd weights situation for Case I and Case III in the proof above. However, we still can consider the Case II, i.e., both of $F$ and $G$ are non-liftings with $k_1$ and $k_2$ being odd integers. In particular, with a similar argument, we can show the following result.
\begin{corollary}
Let $k_1, k_2\in\mathcal{K}^{\mathbf{(G)}}(2)$ be two odd positive integers, where $k_1$ and $k_2$ may equal. Let $F\in\mathcal{S}_{k_1}(\Gamma_2)$ and $G\in\mathcal{S}_{k_2}(\Gamma_2)$ be Hecke eigenforms. If $\lambda_F(2)=\lambda_G(2)$, then $F=c\cdot G$ for some non-zero constant $c$.
\end{corollary}
\begin{remark}
Our approach cannot apply for the case that $k_1$ and $k_2$ have the different parity. It would be interesting to work out a general result of Theorem~\ref{disti-maeda} without any restriction of weights.
\end{remark}
\section{Distinguishing Hecke eigenforms by using \texorpdfstring{$L$}{L}-functions}\label{sect distinguish by L function}

In this section, our main goal is to prove Proposition~\ref{Detemination of SK} for Saito-Kurokawa liftings and Theorem~\ref{dis-non} for non-liftings by using the machine of $L$-functions.
\begin{proof}[Proof of Proposition~\ref{Detemination of SK}] The proof is essentially based on \cite[Theorem~B]{LR1997}. In fact, it suffices to assume that $d<0$.
For a Saito-Kurokawa lifting $F_f$, by \eqref{twisted L-function for SK lift} we have
\begin{equation}
L(1/2,\pi_{F_f}\times\chi_d,\rho_4\otimes\sigma_1)=L(0,\chi_d)L(1,\chi_d)L(1/2,\pi_f\times \chi_d),
\end{equation}
where $\sigma_1$ is the standard representation of the dual group $\C^\times$. By the well-known result of Dirichlet, we have $L(1,\chi_d)\neq 0$. Then by the functional equation of $L(s,\chi_d)$, we get that $L(0,\chi_d)\neq 0$ as well. Moreover, we have $\xi_d(-1)=-1$ since $d<0$. Then by the assumption \eqref{main assumption for G type Thm 3} we obtain that 
\begin{equation}
L(1/2,\pi_f\times \chi_d)=c\cdot L(1/2,\pi_g\times\chi_d)
\end{equation}
for almost all quadratic Hecke characters $\chi_d$ of $\Q^\times\backslash\A^\times$, which are corresponding to primitive Dirichlet quadratic characters $\xi_d$ of conductor $d$. Recall that $f$ is of weight $2k_1-2$ and $g$ is of weight $2k_2-2$ with $k_1, k_2$ even, then the root numbers of the cuspidal automorphic  representation associated to $f$ and $g$ are $-1$. Similar to \cite[(3.1)]{LR1997}, we also have the set 
\begin{equation}
\mathcal{D}^{\omega}=\{d\in\Z\colon \mbox{$\omega d>0, d\equiv v^2 \pmod{4M}$, for some $v$ coprime to $4M$ and $M$ is an integer}\},
\end{equation}
where $\omega$ is the root number. This is exactly our case since we assume that $d<0$ and $\omega=-1$.
In this case, for any $d\in\mathcal{D}^{\omega}$, we can find a non-zero constant $c$ such that
\begin{equation}
L(1/2,\pi_f\times\chi_d)=c\cdot L(1/2,\pi_g\times\chi_d).
\end{equation}
By the virtual of \cite[Theorem~B]{LR1997}, we have $k_1=k_2$ and $f=g$. Therefore, $F_f=F_g$ as desired. 
\end{proof}

\begin{proof}[Proof of Theorem~\ref{dis-non}]
We would consider the following integral
\begin{equation}
\frac{1}{2\pi i}\int_{(2)}\left(\frac{x^{s-\frac{1}{2}}-x^{\frac{1}{2}-s}}{s-\frac{1}{2}}\right)^2\left(-\frac{Z'}{Z}(s)\right)\,ds,
\end{equation}
where later we will choose $Z(s)$ to be $L(s,\pi_F\times \pi_F,\rho_4\otimes\rho_4)$ and $L(s,\pi_F\times \pi_G,\rho_4\otimes\rho_4)$, respectively. Assume that 
\begin{equation}
-\frac{L'}{L}(s,\pi_F\times \pi_F,\rho_4\otimes\rho_4)=\sum_{n=1}^{\infty}\frac{\Lambda_{F\times F}(n)}{n^s}\quad\text{and}\quad -\frac{L'}{L}(s,\pi_F\times \pi_G,\rho_4\otimes\rho_4)=\sum_{n=1}^{\infty}\frac{\Lambda_{F\times G}(n)}{n^s}.    
\end{equation}
Following the idea of \cite{GoHo1993}, for $x>0$ we can show that 
\begin{equation}\label{sum of Lambda F F}
 2\sum_{n<x^2}\frac{\Lambda_{F\times F}(n)}{n^{\frac{1}{2}}}\log\left(\frac{x^2}{n}\right)=8(x-2+x^{-1})-4\sum_{\gamma}\frac{\sin^2(\gamma\log x)}{\gamma^2}+J_1,   
\end{equation}
where $\frac{1}{2}+i\gamma$ runs over the non-trivial zeros of $L(s,\pi_F\times \pi_F,\rho_4\otimes\rho_4)$ and
\begin{equation}
J_1=\frac{1}{2\pi i}\int_{(1/2)}\left(\frac{G_1'}{G_1}(s)+\frac{G_1'}{G_1}(1-s)\right)\left(\frac{x^{s-\frac{1}{2}}-x^{\frac{1}{2}-s}}{s-\frac{1}{2}}\right)^2\,ds.
\end{equation}
Here, $G_1(s)$ is the archimedean part of $L(s,\pi_F\times \pi_F,\rho_4\otimes\rho_4)$; see Proposition~\ref{gammaspintensor} with $k_1=k_2$. Note that we moved the integration line to $\re(s)=1/2$ since there exist no poles of $G_1(s)$ when $1/4\leq\re(s)\leq 3/4$ by Proposition~\ref{gammaspintensor} in Appendix~\ref{Archimedean factors append}.

Similarly, in the case of $L(s,\pi_F\times \pi_G,\rho_4\otimes\rho_4)$, we have  
\begin{equation}\label{sum of Lambda F G}
    2\sum_{n<x^2}\frac{\Lambda_{F\times G}(n)}{n^{\frac{1}{2}}}\log\left(\frac{x^2}{n}\right)=-4\sum_{\gamma'}\frac{\sin^2(\gamma'\log x)}{(\gamma')^2}+J_2,
\end{equation}
where $\frac{1}{2}+i\gamma'$ runs over the non-trivial zeros of $L(s,\pi_F\times \pi_G,\rho_4\otimes\rho_4)$ and
\begin{equation}
J_2=\frac{1}{2\pi i}\int_{(1/2)}\left(\frac{G_2'}{G_2}(s)+\frac{G_2'}{G_2}(1-s)\right)\left(\frac{x^{s-\frac{1}{2}}-x^{\frac{1}{2}-s}}{s-\frac{1}{2}}\right)^2\,ds.
\end{equation}
Here, $G_2(s)$ is the archimedean part of $L(s,\pi_F\times \pi_G,\rho_4\otimes\rho_4)$. By \cite[Proposition~5.7]{IwKo2004}, we can show that
\begin{equation}
\sum_{\gamma}\frac{\sin^2(\gamma\log x)}{\gamma^2},\hspace{4mm}\sum_{\gamma'}\frac{\sin^2(\gamma'\log x)}{(\gamma')^2}\ll\log(k_1k_2)(\log x)^2.
\end{equation}
By Stirling's formula, we can show that
\begin{equation}
J_1, J_2\ll O(\log(k_1k_2)(\log x)^2).
\end{equation}
Suppose that $\Lambda_{F\times F}(n)=\Lambda_{F\times G}(n)$ for all $n<x^2$. Subtracting \eqref{sum of Lambda F G} from \eqref{sum of Lambda F F} implies
\begin{equation}
0=8(x-2+x^{-1})+O((\log k_1k_2)(\log x)^2).
\end{equation}
This will give a contradiction when $x\gg(\log k_1k_2)(\log\log k_1k_2)^2$. That is, if $F$ is not a multiple of $G$, then we can find a sufficiently large $C$ such that, for some integer $n\leq C(\log k_1k_2)^2(\log\log k_1k_2)^4$, $\Lambda_{F\times F}(n)\neq \Lambda_{F\times G}(n)$. Then Theorem~\ref{dis-non} can be deduced by Lemma~\ref{rkselemma} below.
\end{proof}

\begin{lemma}\label{rkselemma}
Assume the notations above. Suppose that we can find $A$ such that $\Lambda_{F\times F}(n)\neq\Lambda_{F\times G}(n)$ for some $n\leq A$. Then we can find $n\leq A$ such that $a_F(n)\neq a_G(n)$. Moreover, for such $n\leq A$ we have $\tilde{\lambda}_F(n)\neq \tilde{\lambda}_G(n)$.
\end{lemma}
\begin{proof}
As for the first assertion, notice that $\Lambda_{F\times F}(n)$ and $\Lambda_{F\times G}(n)$ are arithmetic functions supported on prime powers. So there exist a prime number $p$ and a positive integer $r$ such that $p^r\leq A$, and $\Lambda_{F\times F}(p^r)\neq \Lambda_{F\times G}(p^r)$. We consider two cases: when $p\geq \lfloor A^{1/2}\rfloor+1$ and $p\leq \lfloor A^{1/2}\rfloor$. Here, $\lfloor x\rfloor$ denotes the greatest integer less than or equal to $x$.

In the first case, $p^2>A$ and hence  $\Lambda_{F\times F}(p)\neq \Lambda_{F\times G}(p)$. It can be shown that $\Lambda_{F\times F}(p)=a_F(p)^2\log p$ and $\Lambda_{F\times G}(p)=a_F(p)a_G(p)\log p$. Therefore, we have $a_F(p)\neq a_G(p)$.

In the second case, we prove by contradiction. Suppose that $a_F(n)=a_G(n)$ for $n\leq A$. Then for $p\leq \lfloor A^{1/2}\rfloor$, we have $a_F(p^i)=a_G(p^i)$ for $i=1,2$. This implies that $F$ and $G$ have the same Satake parameters at $p$ (up to permutation), which can be obtained by \eqref{nor-ei-spin} and \eqref{lambda p}-\eqref{lambda p2}. This shows that $\Lambda_{F\times F}(p^r)=\Lambda_{F\times G}(p^r)$ for any $r$, which is a contradiction.

As for the second assertion, it follows from the relation between $a_F(n)$ and $\tilde{\lambda}_F(n)$; see \eqref{nor-ei-spin}.
\end{proof}

\begin{remark}
Suppose that
\begin{equation*}
 L(s,\pi_F,\rho_5)=\sum_{n=1}^{\infty}\frac{b_F(n)}{n^s}\quad\text{and}\quad L(s,\pi_G,\rho_5)=\sum_{n=1}^{\infty}\frac{b_G(n)}{n^s}.  
\end{equation*}
Assume that $L(s,\pi_F\times \pi_G,\rho_5\otimes\rho_5)$ and $L(s,\pi_F\times \pi_F,\rho_5\otimes\rho_5)$ satisfies the Generalized Riemann Hypothesis. A similar argument will show that if $F$ is not a scalar multiplication of $G$, then there exists an integer 
\[n\ll(\log k_1k_2)^2(\log\log k_1k_2)^4\]
such that $b_F(n)\neq b_G(n)$. Indeed, a direct calculation will show that, $\{b_F(p^r)\}_{r=1}^{\infty}$ will determined by $\{b_F(p),b_F(p^2)\}$ and hence we can obtain a result similar to the first assertion as in Lemma~\ref{rkselemma}.
\end{remark}

\begin{appendix}
\section{Archimedean factors associated to certain \texorpdfstring{$L$}{L}-functions}\label{Archimedean factors append}
In this section, we briefly discuss the calculation of the archimedean factors associated to the $L$-functions shown up in the proof of Theorem~\ref{dis-non} as in Section~\ref{sect distinguish by L function}.

Let $F\in\mathcal{S}_{k_1}(\Gamma_2)$ and $G\in\mathcal{S}_{k_2}(\Gamma_2)$ be Hecke eigenforms. Then we can associate the cuspidal automorphic  representations $\pi_F$ (resp. $\pi_G$) for $F$ (resp. $G$) of $\GSp(4, \A)$. For $\pi_F$, we can associate the completed spinor $L$-function and the completed standard $L$-function, denoted by $\Lambda(s,\pi_F,\rho_4)$ and $\Lambda(s,\pi_F,\rho_5)$, respectively. Moreover, via the Langlands transfer (see \cite[\S~5.1]{PiSaSc2014}), we can find $\Pi_4^F$ (resp. $\Pi_5^F$), which is a cuspidal automorphic representation of $\GL(4,\A)$ (resp. $\GL(5,\A)$) such that
\[\Lambda(s,\pi_F,\rho_4)=\Lambda(s,\Pi_4^F)\quad\text{and}\quad \Lambda(s,\pi_F,\rho_5)=\Lambda(s,\Pi_5^F).\]
In this case, the Rankin-Selberg $L$-function $\Lambda(s,\pi_F\times\pi_G,\rho_4\otimes\rho_4)$ and $\Lambda(s,\pi_F\times\pi_G,\rho_5\otimes\rho_5)$ is defined by the Rankin-Selberg convolutions on $\GL(4)\times\GL(4)$ and $\GL(5)\times\GL(5)$, respectively, i.e.,
\begin{equation}
 \Lambda(s,\pi_F\times\pi_G,\rho_4\otimes\rho_4)=\Lambda(s,\Pi_4^F\times\Pi_4^G)\quad\text{and}\quad \Lambda(s,\pi_F\times\pi_G,\rho_5\otimes\rho_5)=\Lambda(s,\Pi_5^F\times\Pi_5^G).  
\end{equation}
To calculate the associated archimedean factors, we recall some basic facts regarding the real Weil group $W_{\mathbb{R}}=\mathbb{C}^\times\sqcup j\mathbb{C}^\times$. Here, the multiplication on $\C^\times$ is standard, and $j$ is an element satisfying $j^2=-1$ and $jzj^{-1}=\bar{z}$ (complex conjugation) for $z\in\C^\times$. More precisely, we are considering representations of $W_\R$, which are continuous homomorphisms $W_\R\to \GL(n, \C)$ for some $n$ with the image consisting of semisimple elements. By \cite{Knapp1994}, every finite-dimensional semisimple representation of $W_\R$ is completely recucible, and each irreducible representation is either one- or two-dimensional. The complete list of one-dimensional representations is as follows:
\begin{align}
   \varphi_{+,t}&\colon re^{i\theta}\longmapsto r^{2t},\quad  j\mapsto 1,\\
   \varphi_{-,t}&\colon re^{i\theta}\longmapsto r^{2t},\quad  j\mapsto -1,
\end{align}
where $t\in\C$, and we write any non-zero complex number $z$ as $re^{i\theta}$ with $r\in \R_{>0}$ and $\theta\in \R/2\pi\Z$. The two-dimensional representations are precisely
\begin{equation}
    \varphi_{\ell, t}\colon re^{i\theta}\mapsto \left[\begin{smallmatrix}
r^{2t}e^{i\ell\theta}\\
&r^{2t}e^{-i\ell\theta}
\end{smallmatrix}\right],\qquad j\mapsto
\left[\begin{smallmatrix}
&(-1)^{\ell}\\
1
\end{smallmatrix}\right],
\end{equation}
where $\ell\in\Z_{>0}$ and $t\in\C$. And the corresponding $L$-factors, i.e., the archimedean factors, are given as follows:
\begin{equation}
    L_\infty(s, \varphi)=\begin{cases}
    \Gamma_\R(s+t)&\mbox{if $\varphi=\varphi_{+, t}$},\\
    \Gamma_\R(s+t+1)&\mbox{if $\varphi=\varphi_{-, t}$},\\
    \Gamma_\C(s+t+\frac{\ell}{2})&\mbox{if $\varphi=\varphi_{\ell, t}$}.
    \end{cases}
\end{equation}
Here,
\begin{equation}
    \Gamma_\R(s)\coloneqq\pi^{-s/2}\Gamma\left(\frac{s}{2}\right),\qquad \Gamma_\C(s)\coloneqq 2(2\pi)^{-s}\Gamma(s),
\end{equation}
where $\Gamma(s)$ is the usual gamma function. By a direct calculations we have the following lemma:
\begin{lemma}\label{Tensor product lemma}
For $\ell, \ell_1, \ell_2\in\Z_{>0}$ and $t_1,t_2\in\C$, we have
\begin{align}
\varphi_{+, t_1}\otimes\varphi_{+, t_2}=\varphi_{-, t_1}\otimes\varphi_{-, t_2}=&\varphi_{+, t_1+t_2}\label{++--}\\
\varphi_{+, t_1}\otimes\varphi_{-, t_2}=\varphi_{-, t_1}\otimes\varphi_{+, t_2}=&\varphi_{-, t_1+t_2}\label{+--+}\\
\varphi_{\pm, t_1}\otimes\varphi_{\ell, t_2}=&\varphi_{\ell, t_1+t_2}\\
\varphi_{\ell_1, t_1}\otimes\varphi_{\ell_2, t_2}=&\begin{cases}
\varphi_{\ell_1+\ell_2, t_1+t_2}\oplus\varphi_{|\ell_1-\ell_2|, t_1+t_2}&\text{if } \ell_1\neq\ell_2,\\
\varphi_{\ell_1+\ell_2, t_1+t_2}\oplus\varphi_{+, t_1+t_2}\oplus\varphi_{-, t_1+t_2}&\text{if } \ell_1=\ell_2.
\end{cases}\label{2tensor2}
\end{align}
\end{lemma}
\begin{remark}\label{rk ell1=ell2 holds as well}
Recall that $\Gamma_\C(s)=\Gamma_\R(s)\Gamma_\R(s+1)$, the second case in \eqref{2tensor2} looks precisely like the first case in \eqref{2tensor2}, if we allow $\ell_1=\ell_2$. \end{remark}
For our purpose, we will only consider the case $t=0$; in this case, we write $\varphi_{\pm}$ instead of $\varphi_{\pm, 0}$ and $\varphi_\ell$ instead of $\varphi_{\ell, 0}$. It follows from \cite[\S~3.2]{Schmidt2017} (observing that $\lambda_1=k_1-1$ and $\lambda_2=k_1-2$) and \cite[Theorem~5.1.2]{PiSaSc2014} that the $L$-parameter of $\Pi_4^F$ at the archimedean place is given by:
\begin{equation}\label{degree 4 RS L factor}
\varphi_{2k_1-3}\oplus\varphi_{1}.  
\end{equation}

Then we have the following proposition proved by Gun, Kohnen and Paul in \cite[pp.~56-57]{GunKohnenPaul2021}. In their proof, they considered two separated cases ($k_1>k_2$ and $k_1=k_2$). Observing Remark~\ref{rk ell1=ell2 holds as well}, these two cases can be combined as follows:
\begin{proposition}\label{gammaspintensor}
Assume the notations above. The archimedean factor of $\Lambda(s,\pi_F\times\pi_G,\rho_4\otimes\rho_4)$ is given by:
\begin{equation*}
\begin{split}
&\Gamma_\C\left(s+k_1+k_2-3\right)\Gamma_\C\left(s+k_1-1\right)\Gamma_\C(s+k_2-1)\Gamma_\C\left(s+k_1-2\right)\Gamma_\C\left(s+k_2-2\right)\\
&\Gamma_\C(s+1)\Gamma_\C(s+|k_1-k_2|)\Gamma_\R(s)\Gamma_\R(s+1).
\end{split}
\end{equation*}
\end{proposition}

Next, we consider the archimedean factor of the Rankin-Selberg $L$-function $\Lambda(s,\pi_F\times\pi_G,\rho_5\otimes\rho_5)$. It follows from the construction of the standard $L$-function that the $L$-parameter of $\Pi_5^F$ is
\begin{equation}\label{degree 5 RS L factor}
\varphi_{2k_1-2}\oplus\varphi_{2k_1-4}\oplus\varphi_+.
\end{equation}
Here, we require that $k_1\geq 2$; see \cite[Table~5]{Schmidt2017}. Then by \eqref{degree 5 RS L factor} and Lemma~\ref{Tensor product lemma} we have
\begin{proposition}
Assume the notations above. The archimedean factor of $\Lambda(s,\pi_F\times\pi_G,\rho_5\otimes\rho_5)$ is given by:
\begin{equation*}
\begin{split}
&\Gamma_\C\left(s+k_1+k_2-2\right)\Gamma_\C\left(s+|k_1-k_2|\right)^2\Gamma_\C\left(s+k_1+k_2-3\right)^2\Gamma_\C\left(s+|k_2-k_1-1|\right)\Gamma_\C\left(s+k_1-1\right)\\
&\Gamma_\C\left(s+|k_1-k_2-1|\right)\Gamma_\C\left(s+k_1+k_2-4\right)\Gamma_\C\left(s+k_1-2\right)\Gamma_\C\left(s+k_2-1\right)\Gamma_\C\left(s+k_2-2\right)\Gamma_\R(s).
\end{split}
\end{equation*}
Again, by Remark~\ref{rk ell1=ell2 holds as well} we can write $\Gamma_\C\left(s+0\right)$ as  $\Gamma_\R\left(s\right)\Gamma_\R\left(s+1\right)$ if happens.
\end{proposition}
\end{appendix}

\bibliographystyle{alpha}
\bibliography{On_distinguishing_Siegel_cusp_forms_of_degree_two.bib}

\vspace{5ex}
\noindent Department of Mathematics, Brown University, Providence, RI 02912, USA.

\noindent E-mail address: {\tt zhining\_wei@brown.edu}

\vspace{2ex}
\noindent School of Mathematical Sciences, Xiamen University, Xiamen, Fujian 361005, China.

\noindent E-mail address: {\tt yishaoyun926@xmu.edu.cn}

\end{document}